\documentclass[a4paper]{scrartcl}
\usepackage[utf8]{inputenc}
\usepackage[T1]{fontenc}
\usepackage[english]{babel}
\usepackage{amsmath,amssymb,amsfonts,siunitx,commath}
\usepackage{amsthm}
\usepackage{mathrsfs}
\usepackage{tikz,url}
\usepackage{geometry}
\usepackage{mathtools}
\mathtoolsset{showonlyrefs}
\usepackage{subcaption}
\usepackage{algorithm}
\usepackage{algpseudocode}
\usepackage{hyperref}
\usepackage{enumerate}
\usepackage{listings}

\newcolumntype{L}{D{.}{.}{2,5}}
\theoremstyle{plain}
\newtheorem{thm}{Theorem}[section]
\newtheorem{lemma}[thm]{Lemma}
\newtheorem{remark}[thm]{Remark}

\newtheorem{corollary}[thm]{Corollary}
\newtheorem{proposition}[thm]{Proposition}

\newcommand{\R}{\mathbb{R}}

\DeclareMathOperator{\Id}{Id}
\DeclareMathOperator{\Fix}{Fix}

\DeclareMathOperator{\CAT}{CAT}
\DeclareMathOperator{\argmin}{argmin}

\DeclareMathOperator{\supp}{supp}
\theoremstyle{definition}

\usepackage{cite}
 \usepackage{hyperref}
 \hypersetup{
    colorlinks=true,
    linkcolor=black,
    filecolor=magenta,      
    urlcolor=cyan}

\author{Arian B\"erd\"ellima\footnotemark[1]\footnote{
		Institute of Mathematics,
		Technische Universit\"at Berlin,
		Stra{\ss}e des 17. Juni 136, 
		D-10623 Berlin, Germany,
		berdellima, steidl@math.tu-berlin.de.
	} 
\and
Gabriele Steidl\footnotemark[1]
}
\title{Quasi $\alpha$-Firmly Nonexpansive Mappings in Wasserstein Spaces}
\date{\today}

\begin{document}
\maketitle

\begin{abstract}	
	This paper  introduces the concept of quasi $\alpha$-firmly nonexpansive mappings
	in Wasserstein spaces over $\R^d$ and analyzes properties of these mappings.
	We prove that for quasi $\alpha$-firmly nonexpansive mappings 
	satisfying a certain quadratic growth condition,
	the fixed point iterations converge in the narrow topology. 
	As a byproduct, we will get the known convergence of the proximal point algorithm in Wasserstein spaces.
	We apply our results to show for the first time that cyclic proximal point algorithms
	for minimizing the sum of certain functionals on Wasserstein spaces converge
	under appropriate assumptions.	
\end{abstract}

\section{Introduction}
Splitting algorithms which include proximal operators 
have recently found broad interest both in Hilbert spaces \cite{BC11} 
and nonlinear CAT(0) spaces \cite{Bacak}.
For certain applications in finite dimensional linear spaces we refer to the overview papers \cite{BPCPE11},
and in Hadamard manifolds to \cite{BacBerSteWei16,BerPerSte16}.
On the other hand, Wasserstein spaces and Wasserstein proximal mappings
are popular in connection with gradient flows \cite{S2017}.

In this paper, we introduce the concept of quasi $\alpha$-firmly nonexpansive mappings in Wasserstein-2 spaces over $\R^d$. For linear spaces such operators were examined in various papers, see, e.g.  \cite{LukTamTha18, Berd-Steidl, BC11}.
The main motivation for studying (quasi) $\alpha$-firmly nonexpansive operators in linear spaces, particularly in Hilbert spaces, is their connection with the so-called averaged operators which are essential in fixed point theory, see, e.g. 
the classical works \cite{Bruck, Browder, Krasnoselski, Mann}. 
In the context of nonlinear CAT(0) spaces (quasi) $\alpha$-firmly nonexpansive mappings were introduced in \cite{Berdellima} 
and later extended in \cite{Luke} to more general settings. 
For $d=1$ the Wasserstein space is CAT(0) 
and the theory about (quasi) $\alpha$-firmly nonexpansive operators follows from \cite{Berdellima}. 
Therefore our theory is a new contribution in the case $d\geq 2$. 
We will see that quasi $\alpha$-firmly nonexpansive mappings in Wasserstein spaces are closed under compositions of operators, 
whenever they share a common fixed point. 
Prominent examples of such mappings are the Wasserstein proximal mappings of proper, lower semicontinuous, coercive functions that are convex along generalized geodesics. 
Also the push-forward mappings of measures by an $\alpha$-firmly nonexpansive operator in $\R^d$ 
constitute an example of its own interest. 
As an application of such mappings, we consider the cyclic proximal point algorithm.  
In contrast to CAT(0) spaces, Wasserstein spaces have a positive Alexandrov curvature \cite{alex} for $d\geq 2$, 
which makes the analysis of algorithms  including these operators in general quite tricky.
Under appropriate conditions we show that the iterations of this algorithm converge 
in the narrow topology to a minimizer of a given finite sum of proper, lower semicontinuous, 
coercive functions that are convex along generalized geodesics. 
Both situations when these functions share or don't share a common minimizer are treated. 
In the latter case, Lipschitz continuity of each constituent function is needed. 
These results have direct relations to finding the minimum of certain energy and relative entropy functionals, 
see \cite[\S9.3, \S9.4]{Ambrosio}. 
It is known that the minima of such functionals are the stationary solutions 
of corresponding stochastic differential equations and that the corresponding density functions 
appear as solutions of partial differential equations as, e.g. the well-examined Fokker--Planck equation, see, e.g., \cite{JKO1998}.

The outline of this paper is as follows:
Section \ref{sec:prelim} contains the basic notation required for our analysis
in Wasserstein spaces.
In Section \ref{sec:prox}, we study proximal mappings 
of functions that are proper, lsc, coercive and convex along generalized geodesics.
Then, in Section \ref{sec:firmly}, we introduce the concept of quasi $\alpha$-firmly nonexpansive mappings
in Wasserstein spaces. We show that proximal mappings of certain functions are quasi $\frac12$-firmly nonexpansive.
Further, we examine push-forward operators of measures for operators on $\R^d$ which are themselves
(quasi) $\alpha$-firmly nonexpansive.
Since $\alpha$-firmly nonexpansive operators in Hilbert spaces are important due to the related fixed point theory,
we examine the fixed point properties of such operators in Wasserstein spaces in Section \ref{sec:fix}. 
As in Hilbert spaces, the path to go is via Opial's property and Fej\'er's monotonicity.
In Section \ref{sec:cppa}, we apply our results to prove the convergence of the cyclic proximal point algorithm.

\section{Preliminaries} \label{sec:prelim}
The following section provides the necessary facts and notation on Wasserstein spaces
as they can be found in several textbooks as \cite{Ambrosio,ABS21,Santambrosio2015,Villani}.
For applications we also refer to \cite{CP2019}.

Let $\mathbb R^d$, $d\geq 1$ be equipped with the Euclidean norm $\|\cdot\|$, 
and let $\mathcal B(\mathbb R^d)$ be its Borel $\sigma$-algebra.
By $\mathcal P_2(\mathbb R^d)$, we denote the set of probability measures on $\mathcal B(\mathbb R^d)$ 
with finite second moments. With the \emph{$L^2$-Wasserstein metric}
\begin{equation} \label{eq:Wasserstein}
W_2(\mu,\nu)\coloneqq\Big(\min_{\pi\in\Pi(\mu,\nu)}\int_{\mathbb R^d\times \mathbb R^d}\|x-y\|^2\,d\pi(x,y)\Big)^{1/2},
\end{equation}
where $\Pi(\mu,\nu)$ denotes the set of all transport plans between $\mu$ and $\nu$, i.e.,
\begin{equation} \label{eq:transportplan}
\pi(A\times \mathbb R^d)=\mu(A)\;\text{and}\;\pi(\mathbb R^d\times B)=\nu(B) \quad \text{for all } A,B\in\mathcal B(\mathbb R^d),
\end{equation}
the space $\mathcal P_2(\mathbb R^d)$  becomes a separable, complete metric space, called Wasserstein space or briefly Wasserstein space.
Let $\Pi_{opt}(\mu,\nu)$ denote the set of optimal transport plans, i.e. the set of all elements in $\Pi(\mu,\nu)$ that attain \eqref{eq:Wasserstein}.

\begin{remark}\label{rem.11}
If $\mu$ is absolutely continuous with respect to the Lebesgue measure, then the  optimal transport plan
$\pi_\mu^\nu$ is unique and is induced by the unique minimizer $T_\mu^\nu$ 
of the so-called Monge problem 
\begin{equation} \label{eq:Monge}
\tilde W_2(\mu,\nu) = \inf_{T}\int_{\mathbb R^d}\|x-T(x)\|^2\,d\mu(x) \quad \text{subject to } \nu = T_\# \mu \coloneqq \mu \circ T^{-1}
\end{equation}
by  $\pi_\mu^\nu = (\Id,T_\mu^\nu)_\# \mu$. 
In this case $\tilde W_2(\mu,\nu)$ coincides with $W_2(\mu,\nu)$.
The situation changes if $\mu$ is not absolutely continuous.
Then, in contrast to the minimization problem \eqref{eq:Wasserstein}, which is also known as  Kantorovich problem, 
the Monge problem
\eqref{eq:Monge} may fail to have a minimizer.
Further, if an optimal transport map $T_\mu^\nu$ in \eqref{eq:Monge} exists, then $\pi \coloneqq (\Id,T_\mu^\nu)_\# \mu \in \Pi(\mu,\nu)$,
i.e. this plan $\pi$ fulfills the marginal conditions.
However, it doesn't have to be optimal as the example 
$\mu := \frac14 \delta_0 + \frac34 \delta_1$ and
$\nu := \frac34 \delta_0 + \frac14 \delta_1$ shows.
\end{remark}

A sequence $(\mu_n)_{n\in\mathbb N}\subset\mathcal P_2(\R^d)$ converges to $\mu\in\mathcal P_2(\R^d)$, denoted 
by $\mu_n\to \mu$, if 
$$\lim_{n\to \infty}W_2(\mu_n,\mu)=0.$$
A sequence $(\mu_n)_{n\in\mathbb N} \subset\mathcal P_2(\R^d)$ \emph{converges narrowly} 
to $\mu\in\mathcal P_2(\R^d)$, denoted  by $\mu_n \overset{\mathcal N}\to \mu$, if 
\begin{equation}\label{eq:narrow-conv}
\int_{\mathbb R^d}\varphi(x)\,d\mu_n(x)\to\int_{\mathbb R^d}\varphi(x)\,d\mu(x) \quad \text{for all } \varphi \in C_b(\R^d).
\end{equation}
The relation between both topologies is given by the following theorem.

\begin{thm}\cite[Theorem 6.9]{Villani} 	\label{l:equiv}
	For $(\mu_n)_{n\in\mathbb N}\subset\mathcal P_2(\R^d)$, we have  $\mu_n \to \mu$ 
	if and only if 
	$\mu_n \overset{\mathcal N}\to\mu$ 
	and 
	$$\int_{\mathbb R^d}\|x\|^2\,d\mu_n(x)\to\int_{\mathbb R^d}\|x\|^2 \,d\mu(x) \quad\text{as } n\to+\infty.$$
\end{thm}

For all $\nu\in\mathcal P_2(\R^d)$, the Wasserstein metric $W_2(\cdot,\nu)$ is lower semicontinuous in the narrow topology, 
i.e. $W_2(\mu,\nu)\leq\liminf_{n\to+\infty}W_2(\mu_n,\nu)$ whenever $\mu_n\overset{\mathcal N}\to\mu$, see \cite[Lemma 4.3]{Villani}.
An important concept is the tightness of a set in $\mathcal P_2(\R^d)$.
A set $S\subseteq \mathcal P_2(\R^d)$ is \emph{tight}, if for every $\varepsilon>0$ 
there exists a compact set $K_{\varepsilon} \subseteq \R^d$ such that 
$\mu(\R^d\setminus K_{\varepsilon})\leq\varepsilon$ for all $\mu\in S$.
  
\begin{thm}[Prokhorov's Theorem \cite{Prokhorov}] 	\label{th:Prokhorov}
	A set $S\subseteq\mathcal P_2(\R^d)$ is tight if and only if $S$ is relatively compact in the topology of narrow convergence.
\end{thm}

In particular, we have the following lemma.

\begin{lemma}\cite[Theorem 1]{Yue} 	\label{l:fundamental}
	Closed balls in $\left(\mathcal P_2(\R^d),W_2 \right)$ are tight. 
\end{lemma}

The Wasserstein spaces is a so-called \emph{geodesic spaces}, 
meaning, that for every $\mu,\nu \in \mathcal P_2(\R^d)$,
there exists a curve $\gamma:[0,1]\to\mathcal P_2(\R^d)$ 
with $\gamma(0) = \mu$, $\gamma(1) = \nu$ and
\begin{equation} \label{geode}
W_2(\gamma(t),\gamma(s)) = |t-s| W_2(\gamma(0),\gamma(1)) \quad \text{for every } t,s\in[0,1].
\end{equation}
A  curve $\gamma:[0,1]\to\mathcal P_2(\R^d)$ with property \eqref{geode} 
is called constant speed \emph{geodesic}.
If $\pi \in \Pi_{opt}(\mu_1,\mu_2)$, then the curve
\begin{align} \label{geo_plan}
\mu_t^{1 \to 2} &\coloneqq g(t,\cdot)_\# \pi,  \quad t \in [0,1],
\end{align}
with $g: [0,1] \times \R^d \times \R^d \to \R^d$, $(t,x_1,x_2) \mapsto (1-t) x_1 + t x_2$
is a constant speed geodesic connecting $\mu_1$ and $\mu_2$.
Conversely, every constant speed geodesic connecting $\mu_1$ and $\mu_2$ has a representation \eqref{geo_plan}
for a suitable $\pi \in \Pi_{opt}(\mu_1,\mu_2)$, see \cite[Theorem 7.2.2]{Ambrosio}.
In particular, if $\mu_1$ is absolutely continuous with respect to the $d$-dimensional Lebesgue measure, 
then, by Remark \ref{rem.11},  
there exists exactly one such constant speed geodesic.

We will need a more general definition of geodesics in order to make the Wasserstein proximal mappings in the next section well-defined. 
For $\mu_0,\mu_1,\mu_2 \in \mathcal P_2(\R^d)$,
let $\Pi(\mu_0, \mu_1, \mu_2)$ denote the set of measures $\boldsymbol{\pi} \in \mathcal P_2(\R^d \times \R^d \times \R^d)$ with marginals
$\mu_i$, $i=0,1,2$, and 
let $\Pi^{j,k}: \R^d \times \R^d \times \R^d \to \R^d \times \R^d$, $(x_0,x_1,x_2) \mapsto (x_j,x_k)$ for $j,k=0,1,2$.
A \emph{generalized geodesic} connecting $\mu_1$ and $\mu_2$ with base $\mu_0$
is any curve of type
$$
\mu_{0,t}^{1 \to 2} \coloneqq h(t,\cdot)_\# {\boldsymbol{\pi}}, 
$$
with
$
h(t,\cdot): [0,1] \times \R^d \times \R^d \times \R^d \to \R^d$, $(t,x_0,x_1,x_2) \mapsto (1-t) x_1 + t x_2
$,
where 
\begin{equation}\label{gg-plan}
\boldsymbol{\pi} \in \Pi(\mu_0, \mu_1, \mu_2), \qquad \Pi^{0,1}_\#\boldsymbol{\pi}  = \Pi_{opt}(\mu_0, \mu_1), 
\quad 
\Pi^{0,2}_\#\boldsymbol{\pi} = \Pi_{opt}(\mu_0, \mu_2).
\end{equation}
Choosing the base $\mu_0 = \mu_1$, we have again the definition of a geodesic.
Moreover, for an absolutely continuous base measure $\mu_0$, 
the generalized geodesic connecting $\mu_1$ and $\mu_2$ 
is uniquely determined and the plan in \eqref{gg-plan} is given via the optimal transport maps by 
$\boldsymbol{\pi}=(\Id,T_{\mu_0}^{\mu_1},T_{\mu_0}^{\mu_2})_\# \mu_0$, e.g. see \cite[Remark 9.2.3]{Ambrosio}.

We consider functions $\mathscr F: \mathcal P_2(\R^d) \to (-\infty,\infty]$ with effective domain
$$D(\mathscr F) \coloneqq \{\mu \in \mathcal P_2(\R^d): \mathscr F(\mu) < \infty\}$$
and call a function \emph{proper} if $D(\mathscr F) \not = \emptyset$.
A function $\mathscr F: \mathcal P_2(\R^d) \to (-\infty,\infty]$ is said to be
\emph{
convex along generalized geodesics}, if for every
$\mu_0,\mu_1,\mu_2 \in D(\mathscr F)$, there exists a generalized geodesic $\mu_{0,t}^{1 \to 2}$ with base $\mu_0$
such that 
\begin{equation} \label{eq:cgg}
\mathscr F(\mu_{0,t}^{1 \to 2}) \le (1-t) \mathscr F(\mu_1) + t \mathscr F(\mu_2) 
\quad \text{for all } t \in [0,1].
\end{equation}

Typical examples of functions defined on $\mathcal P_2(\mathbb R^d)$ that are convex along generalized geodesics are the 
potential and interaction energy and the relative entropy discussed, e.g., in \cite[\S9.3, \S9.4]{Ambrosio}.

\section{Proximal mappings}\label{sec:prox}
In this section, we consider proximal mappings in Wasserstein spaces, which
play an important role in Wasserstein gradient flow methods.
Let $\mathscr F:\mathcal P_2(\R^d) \to(-\infty,+\infty]$ be proper, lower semicontinuous (lsc), coercive (in the sense of \cite[(2.4.10)]{Ambrosio}) and convex along generalized geodesics. 
Then the proximal mapping $\mathscr J_{\tau}: \mathcal P_2(\R^d) \to \mathcal P_2(\R^d)$ given by
\begin{equation} \label{eq:proximal-mapping}
\mathscr J_{\tau}(\mu) \coloneqq \argmin_{\nu\in\mathcal P_2(\R^d) }\Big\{\mathscr F(\nu)+\frac{1}{2\,\tau} W^2_2(\nu,\mu)\Big\},
\quad \tau>0
\end{equation}
is well-defined, i.e., for every $\mu\in\overline{D(\mathscr F)}$, 
the minimizer in \eqref{eq:proximal-mapping} exits and is unique, see  \cite[Theorem 4.1.2, Lemma 9.2.7]{Ambrosio}. 
Moreover, by \cite[Theorem 4.1.2]{Ambrosio} (with $\lambda=0$), for all $\mu \in \overline{D(\mathscr F)}$ and all $\nu\in D(\mathscr F)$, the following inequality is satisfied
\begin{align}
 \label{eq:1}\frac{1}{2\,\tau} W^2_2(\mathscr J_{\tau}(\mu),\nu)-\frac{1}{2\,\tau} W^2_2(\mu,\nu)
\leq  \mathscr F(\nu)- \mathscr F(\mathscr J_{\tau}(\mu))-\frac{1}{2\,\tau}W^2_2(\mathscr J_{\tau}(\mu),\mu).
\end{align}
Replacing $\nu$ with $\mathscr J_{\tau}(\nu)$ and changing the roles of $\mu$ with $\nu$ in \eqref{eq:1}, 
we obtain for all $\mu,\nu\in \mathcal P_2(\R^d)$ that
\begin{align} \label{eq:2}
W^2_2(\mathscr J_{\tau}(\mu),\mathscr J_{\tau}(\nu))
&\leq
\frac{1}{2}\Big(W^2_2(\mu,\mathscr J_{\tau}(\nu))+W^2_2(\mathscr J_{\tau}(\mu),\nu)\\
& \quad - W^2_2(\mathscr J_{\tau}(\mu),\mu)
- W^2_2(\mathscr J_{\tau}(\nu),\nu)\Big).
\end{align}
As in Hilbert spaces, minimizers of $\mathscr F:\mathcal P_2(\R^d) \to(-\infty,+\infty]$
and fixed points of its proximal mapping are closely related.

\begin{proposition}
Let $\mathscr F:\mathcal P_2(\R^d) \to(-\infty,+\infty]$ be proper, lsc,  coercive and convex along generalized geodesics. 
Then it holds
$$
\argmin_{\mu\in\mathcal P(\R^d)} \mathscr F(\mu) =  \Fix\mathscr J_{\tau}.
$$
\end{proposition}


\begin{proof}
Let $\hat \mu$ be a minimizer of $\mathscr F$. Then, for all $\mu\in\mathcal P_2(\R^d)$,
$$
\mathscr F(\hat \mu) + \frac{1}{2\tau}\,W_2^2(\hat \mu,\hat \mu)
= 
\mathscr F(\hat \mu)
\leq 
\mathscr F(\mu)
\leq
\mathscr F(\mu)+\frac{1}{2\tau}\,W_2^2(\mu,\hat \mu) 
$$
implying $\hat \mu = \mathscr J_{\tau}(\hat \mu)$. 
The converse follows immediately from inequality \eqref{eq:1} with $\mu$ replaced by $\hat\mu$ and using $\hat\mu=\mathscr J_{\tau}(\hat\mu)$,
\begin{align*}
0&= \frac{1}{2\,\tau} W^2_2(\mathscr J_{\tau}(\hat\mu),\nu)-\frac{1}{2\,\tau} W^2_2(\hat\mu,\nu)
\\
&\leq  \mathscr F(\nu)- \mathscr F(\mathscr J_{\hat\tau}(\mu))-\frac{1}{2\,\tau}W^2_2(\mathscr J_{\tau}(\hat\mu),\hat\mu)=\mathscr F(\nu)- \mathscr F(\hat\mu),
\end{align*}
i.e. $\mathscr F(\hat\mu)\leq \mathscr F(\nu)$ for all $\nu\in D(\mathscr F)$. 
\end{proof}

\section{Quasi $\alpha$-firmly nonexpansive mappings}\label{sec:firmly}
Recall that for $\alpha \in (0,1)$,
an operator $T:\mathbb R^d\to \mathbb R^d$ is \emph{$\alpha$-firmly nonexpansive}, if we have 
\begin{equation} \label{eq:fne-operator}
\|Tx-Ty\|^2\leq\|x-y\|^2-\frac{1-\alpha}{\alpha}\|(\Id-T)x-(\Id-T)y\|^2 \quad \text{for all } x,y \in \R^d.
\end{equation}
If the fixed point set $\Fix T \coloneqq \{x \in \R^d: T(x) = x\}$ is nonempty 
and \eqref{eq:fne-operator} is restricted to $y \in \Fix F$, i.e.,
\begin{equation} \label{eq:quasi-fne-operator}
\|Tx-y\|^2\leq\|x-y\|^2-\frac{1-\alpha}{\alpha}\|(\Id-T)x\|^2 \quad \text{for all } x  \in \R^d,\, y\in\Fix T,
\end{equation}
then $T$ is called \emph{quasi $\alpha$-firmly nonexpansive}.
We do not know how to translate the definition of $\alpha$-firmly nonexpansive operators on $\R^d$ to Wasserstein spaces.
However, for quasi $\alpha$-firmly nonexpansive operators this is possible.
First, we say as usual that
a mapping $\mathscr T:\mathcal P_2(\R^d)\to \mathcal P_2(\R^d)$ is \emph{nonexpansive}, if 
\begin{equation} \label{nonexp}
W_2(\mathscr T(\mu),\mathscr T(\nu))\leq W_2(\mu,\nu)\quad \text {for all }\mu,\nu\in\mathcal P_2(\R^d).
\end{equation}
If the fixed point set 
$\Fix \mathscr T \coloneqq \{\mu\in\mathcal P_2(\R^d):\mathscr T(\mu)=\mu\}$ 
is nonempty 
and \eqref{nonexp} holds for all $\mu\in\mathcal P_2(\R^d)$ and all
$\nu\in\Fix \mathscr T$, 
then $\mathscr T$ is said to be a \emph{quasi nonexpansive mapping}. 
Finally, for $\alpha\in(0,1)$, a mapping $\mathscr T:\mathcal P_2(\R^d)\to\mathcal P_2(\R^d)$ 
is \emph{quasi $\alpha$-firmly nonexpansive}, 
if $\Fix\mathscr T \neq \emptyset$ and 
for  all 
$\mu\in\mathcal P(\R^d)$, $\nu\in\Fix\mathscr T$, the following inequality holds true:
\begin{equation} \label{eq:firmly}
W^2_2(\mathscr T(\mu),\nu) \leq W_2^2(\mu,\nu)-\frac{1-\alpha}{\alpha}\,W_2^2(\mu,\mathscr T(\mu)).
\end{equation}
By the next proposition, proximal mappings are quasi $\alpha$-firmly nonexpansive.

\begin{proposition} 	\label{p:proximal-quasi}
	Let $\mathscr F:\mathcal P_2(\R^d) \to(-\infty,+\infty]$ be proper, lsc, coercive and convex along generalized geodesics. 
	Then, for every $\tau>0$, the proximal mapping 
	$\mathscr J_{\tau}:\mathcal P_2(\R^d)\to\mathcal P_2(\R^d)$  
	is quasi $\alpha$-firmly nonexpansive on $\overline{D(\mathscr F)}$	with $\alpha=1/2$.
\end{proposition}

\begin{proof}
The claim follows immediately from \eqref{eq:2} using $\mathscr J_{\tau} (\nu) = \nu$.
\end{proof}

Next, we are interested in the behavior of push-forward operators of (quasi) $\alpha$-firmly nonexpansive operators $T: \R^d \to \R^d$.
In other words, we consider $\mathscr T_T:\mathcal P_2(\R^d)\to\mathcal P_2(\R^d)$ defined by 
\begin{equation}
\label{eq:operator-T}
\mathscr T_T (\mu) \coloneqq T_{\#}\mu = \mu \circ T^{-1}.
\end{equation}

\begin{proposition} \label{p:operator-T}
	Let $T: \mathbb R^d\to  \mathbb R^d$ be an $\alpha$-firmly nonexpansive operator for a certain $\alpha\in(0,1)$. 
	Then the operator $\mathscr T_T:\mathcal P_2(\R^d)\to\mathcal P_2(\R^d)$ in \eqref{eq:operator-T} is nonexpansive. 
	In particular, with $\widetilde T \coloneqq \Id-T$, it satisfies 
	\begin{equation} 	\label{eq:T-ineq}
	W_2^2(\mathscr T_T(\mu),\mathscr T_T(\nu))
	\leq 
	W^2_2(\mu,\nu)-\frac{1-\alpha}{\alpha}W^2_2(\mathscr T_{\widetilde T}(\mu),\mathscr T_{\widetilde T}(\nu))\quad \text{for all } \mu,\nu\in\mathcal P_2(\R^d).
	\end{equation}
\end{proposition}

\begin{proof}
	Let $\pi\in\Pi_{opt}(\mu,\nu)$. Then 
	\begin{align*}
	W_2^2(\mathscr T_T(\mu),\mathscr T_T(\nu))
	\leq
	\int_{ \mathbb R^d\times  \mathbb R^d}\|x-y\|^2\,d \left((T,T)_{\#}\pi \right)(x,y)=\int_{ \mathbb R^d\times  \mathbb R^d}\|Tx-Ty\|^2\,d\pi(x,y).
	\end{align*}
	By assumption the operator $T: \mathbb R^d\to  \mathbb R^d$ is $\alpha$-firmly nonexpansive for some $\alpha\in(0,1)$, so that
	\begin{align*}
	\int_{ \mathbb R^d\times  \mathbb R^d}&\|Tx-Ty\|^2\,d\pi(x,y)\\&\leq \int_{ \mathbb R^d\times  \mathbb R^d}\|x-y\|^2\,d\pi(x,y)-\frac{1-\alpha}{\alpha}\int_{ \mathbb R^d\times  \mathbb R^d}\|(\Id-T)x-(\Id-T)y\|^2\,d\pi(x,y)\\&
	=W^2_2(\mu,\nu)-\frac{1-\alpha}{\alpha}\int_{ \mathbb R^d\times  \mathbb R^d}\|x-y\|^2\,d\left( (\widetilde T,\widetilde T)_{\#}\pi \right)(x,y)\\
	&
	\leq W^2_2(\mu,\nu)-\frac{1-\alpha}{\alpha}\,W^2_2(\mathscr T_{\widetilde T}(\mu),\mathscr T_{\widetilde T}(\nu)).
	\end{align*}
\end{proof}

The next result describes the relationships between $\Fix T$ and $\Fix \mathscr T_T$.

\begin{proposition} 	\label{p:fix}
	Let $T: \mathbb R^d\to  \mathbb R^d$ be a quasi $\alpha$-firmly nonexpansive operator. 
	Then $\Fix\mathscr T_T \neq\emptyset$ and in particular $\nu \in \Fix\mathscr T_T$ if $\supp(\nu) \subseteq \Fix T$.
\end{proposition}

\begin{proof}
	Let $T: \mathbb R^d\to  \mathbb R^d$ be quasi $\alpha$-firmly nonexpansive mapping.
	By \cite[Lemma 4.1]{Berd-Steidl} it follows that $\Fix T$ is nonempty, closed and convex. 
	Let $\mu\in \mathcal P_2(\R^d)$ be arbitrary. 
	Let $f(x)\coloneqq P_{\Fix T}(x)$ for $x\in\mathbb R^d$,
	where $P_{\Fix T}$ denotes the metric projection onto $\Fix T$.
	Then we have for $\nu \coloneqq f_{\#}\mu$ that $\supp(\nu)\subseteq\Fix T$. 
 If a measure $\nu$ fulfills the later condition, then
	we obtain for
	every $B\in\mathcal B( \mathbb R^d)$ that 
	$$\mathscr T_T(\nu)(B) = (T_{\#}\nu)(B)=\nu(T^{-1}(B))=\nu(T^{-1}(B)\cap\Fix T)=\nu(B\cap \Fix T)=\nu(B),$$
	i.e. $\nu\in\Fix \mathscr T_T$. 
	It remains to show that $\nu\in\mathcal P_2(\R^d)$.
	By definition of $f$, we have  for any $B\in\mathcal B( \mathbb R^d)$ that
		$\nu(B)=\mu(f^{-1}(B))=\mu(f^{-1}(B\cap\Fix T))$. 
		It is evident that $\nu(B)\geq 0$ and that $\nu(\R^d)=\mu(f^{-1}(\Fix T))=\mu( \mathbb R^d)=1$. 
		Moreover, if $(B_n)_{n\in\mathbb N}$ is a countable family of disjoint Borel sets, 
		then so is the family $(B_n\cap\Fix T)_{n\in\mathbb N}$ and consequently
	\begin{align*}
	\nu(\bigcup_{n\in\mathbb N}B_n)&=\mu(f^{-1}(\bigcup_{n\in\mathbb N}B_n\cap \Fix T))\\&=\mu(\bigcup_{n\in\mathbb N}f^{-1}(B_n\cap\Fix T))=\sum_{n\in\mathbb N}\mu(f^{-1}(B_n\cap\Fix T))=\sum_{n\in\mathbb N}\nu(B_n).
	\end{align*} 
	Therefore $\nu$ is indeed a probability measure. 
	Next, consider
	\begin{align*}
	\int_{ \mathbb R^d}\|x\|^2\,d\nu(x)&=\int_{\Fix T}\|x\|^2\,d\mu(x)+\int_{ \mathbb R^d\setminus \Fix T}\|P_{\Fix T}x\|^2\,d\mu(x).
	\end{align*}	
	The second integral can be estimated with an arbitrary fixed $x_0\in\Fix T$ as follows:	
	\begin{align*}
	\int_{ \mathbb R^d\setminus \Fix T}\|P_{\Fix T}x\|^2\,d\mu(x)&\leq \int_{ \mathbb R^d\setminus \Fix T}(\|P_{\Fix T}x-x_0\|+\|x_0\|)^2\,d\mu(x)\\&
	\leq 2\,\Big(\int_{ \mathbb R^d\setminus \Fix T}\|P_{\Fix T}x-x_0\|^2\,d\mu(x)+\int_{ \mathbb R^d\setminus \Fix T}\|x_0\|^2\,d\mu(x)\Big)\\&
	\leq 2\,\Big(\int_{ \mathbb R^d\setminus \Fix T}\|x-x_0\|^2\,d\mu(x)+\int_{ \mathbb R^d\setminus \Fix T}\|x_0\|^2\,d\mu(x)\Big).
	\end{align*}	
	Since $\mu\in\mathcal P_2(\mathbb R^d)$ this completes the proof.	
\end{proof}
In particular we have showed the following relation.

\begin{corollary}
	If $T:\mathbb R^d\to\mathbb R^d$ is a quasi $\alpha$-firmly nonexpansive mapping, 
	then it holds $(P_{\Fix T})_{\#}\mathcal P_2(\mathbb R^d)\subseteq\Fix \mathscr T_T$.
\end{corollary}

Let $\mu,\nu\in\mathcal P_2(\R^d)$ and $\pi\in\Pi(\mu,\nu)$. 
Let $\{\nu_x\}_{x\in \mathbb R^d}$ be the family of disintegrations of $\pi$ with respect to $\mu$ , i.e.
\begin{equation} \label{eq:disintegration}
\pi(A\times B)=\int_A\Big(\int_B\nu_x(y)\,dy\Big)\,d\mu(x) \quad\text{for all } A,B\in\mathcal B(\mathbb R^d),
\end{equation}

\begin{proposition}	\label{p:fix2}
	Let $T:\mathbb R^d\to \mathbb R^d$ be a quasi $\alpha$-firmly nonexpansive operator. 
	Suppose that for every $\nu\in\Fix\mathscr T_T$ there is 
	a family of disintegrating measures $\{\nu_x\}_{x\in \mathbb R^d}$ satisfying $\nu_x(\Fix T)\geq C_{\Fix T}$ 
	for some positive uniform constant $C_{\Fix T}$ possibly depending on $\Fix T$.
	Then $\mathscr T_T:\mathcal P_2(\R^d)\to\mathcal P_2(\R^d)$ is a quasi $\alpha$-firmly nonexpansive operator.
\end{proposition}

\begin{proof}
	Let $\mu\in\mathcal P_2(\R^d)$ and $\nu\in\Fix\mathscr T_T$ satisfying $\nu_x(\Fix T)\geq C_{\Fix T}$ 
	uniformly for some positive constant $C_{\Fix T}$.
	Let $\pi\in\Pi_{opt}(\mu,\nu)$. 
From Proposition \ref{p:fix} we have that $\supp(\nu)\subseteq\Fix T$, therefore
	\begin{align*}
	W_2^2(\mathscr T_T(\mu),\mathscr T_T(\nu))&\leq
	\int_{\mathbb R^d\times \Fix T}\|Tx-Ty\|^2\,d\pi(x,y)\\&\leq \int_{\mathbb R^d\times \Fix T}\|x-y\|^2\,d\pi(x,y)
	\quad -\frac{1-\alpha}{\alpha}\int_{\mathbb R^d\times \Fix T}\|x-Tx\|^2\,d\pi(x,y).
	\end{align*}
	From the assumption on the disintegration we estimate the second integral from below as
	\begin{align*}
	\int_{\mathbb R^d\times \Fix T}\|x-Tx\|^2\,d\pi(x,y)
	&=\int_{\mathbb R^d}\|x-Tx\|^2\Big(\int_{\Fix T}\nu_x(y)\,dy\Big)\,d\mu(x)
	\\&=\int_{\mathbb R^d}\|x-Tx\|^2\nu_x(\Fix T)\,d\mu(x)
	\geq C_{\Fix T}\int_{\mathbb R^d}\|x-Tx\|^2\,d\mu(x).
	\end{align*}
	Together with $\nu\in\Fix\mathscr T_T$, this implies
	\begin{align*}
	W_2^2(\mathscr T_T(\mu),\mathscr T_T(\nu))=	W_2^2(\mathscr T_T(\mu),\nu)&\leq W_2^2(\mu,\nu)-\frac{1-\alpha}{\alpha}\,C_{\Fix T}\int_{\mathbb R^d}\|x-Tx\|^2\,d\mu(x)\\
	&\leq W_2^2(\mu,\nu)-\frac{1-\alpha}{\alpha}\,C_{\Fix T}\,W_2^2(\mu,\mathscr T_T(\mu)).
	\end{align*}
	Taking $\hat\alpha = (1+C_{\Fix T}(1-\alpha)/\alpha)^{-1}$, this completes the proof.
\end{proof}

In the particular case when $\Fix T$ consists of a unique element $x_0\in \mathbb R^d$, 
the previous result holds without any disintegration condition as shown below.

\begin{corollary}
	\label{p:unique-fix}
	Let $T:\mathbb R^d\to \mathbb R^d$ be a quasi $\alpha$-firmly nonexpansive operator such that $\Fix T=\{x_0\}$ for some $x_0\in\mathbb R^d$. 
	Then $\nu_0 \coloneqq \delta_{x_0} \in\mathcal P_2(\mathbb R^d)$ is a fixed point of the push-forward operator $\mathscr T_T$ and fulfills
	$$
	W^2_2(\mathscr T_T(\mu),\mathscr T_T(\nu_0)) \leq W^2_2(\mu,\nu_0)-\frac{1-\alpha}{\alpha}\,W^2_2(\mu,\mathscr T_T(\mu)).
	$$
\end{corollary}

\begin{proof}
	Let $\Fix T=\{x_0\}$ for some $x_0\in\mathbb R^d$. Then, by the proof of Proposition \ref{p:fix}, 
	there is $\nu_0\in\Fix\mathscr T_T$ with $\supp(\nu_0)\subseteq \{x_0\}$, i.e. $\nu_0=\delta_{x_0}$. 
	For any $\mu\in\mathcal P_2(\mathbb R^d)$ the only transport plan between $\mu$ and $\nu_0$ is 
	$\pi=\mu\otimes \delta_{x_0}$. By similar calculations as in Proposition \ref{p:fix2} we obtain 
	\begin{align*}
	W^2_2(\mathscr T_T(\mu),\mathscr T_T(\nu_0))
	&\leq \int_{\mathbb R^d\times\{x_0\}}\|Tx-Tx_0\|^2\,d\pi(x,x_0)=\int_{\mathbb R^d}\|Tx-x_0\|^2\,d\mu(x)\\
	&\leq\int_{\mathbb R^d}\|x-x_0\|^2\,d\mu(x) -\frac{1-\alpha}{\alpha}\,\int_{\mathbb R^d}\|Tx-x\|^2\,d\mu(x)\\&
	\leq W^2_2(\mu,\nu_0)-\frac{1-\alpha}{\alpha}\,W^2_2(\mu,\mathscr T_T(\mu)).
	\end{align*}
\end{proof}

We conclude this section by showing that quasi $\alpha$-firmly nonexpansiveness 
is well behaved under the composition of mappings that share at least a common fixed point. 
The proofs are modifications of arguments from \cite{Berdellima} for our setting.

\begin{lemma}	\label{l:intersections}
	Let $\mathscr S,\mathscr T:\mathcal P_2(\R^d)\to \mathcal P_2(\R^d)$ satisfy $\Fix \mathscr T\cap\Fix \mathscr S\neq\emptyset$.
	If $\mathscr S$ is quasi $\alpha$-firmly nonexpansive and 
	$\mathscr T$ is quasi nonexpansive, then 
	$\Fix (\mathscr T \circ \mathscr S)=\Fix \mathscr T\cap \Fix \mathscr S$.
\end{lemma}

\begin{proof} 
The inclusion $\Fix \mathscr T\cap \Fix \mathscr S\subseteq \Fix (\mathscr T \circ \mathscr S)$ is obvious. 
Now let $\mu\in\Fix (\mathscr T\circ\mathscr S)$. First, suppose that $\mathscr S(\mu)\in \Fix \mathscr T$. 
Then $\mathscr S(\mu)=\mathscr T(\mathscr S(\mu))=\mu$ implies $\mu\in \Fix \mathscr T\cap \Fix \mathscr S$. Next, let $\mathscr S(\mu)\notin\Fix \mathscr T$.
Then we distinguish two subcases $\mu\in \Fix \mathscr S$ and $\mu\notin\Fix \mathscr S$.
If $\mu\in \Fix \mathscr S$, then $\mu=\mathscr T(\mathscr S(\mu))=\mathscr T(\mu)$ 
implies $\mu\in \Fix \mathscr T\cap \Fix \mathscr S$. 
Finally, let $\mu\notin \Fix \mathscr S$ and take $\nu\in \Fix \mathscr T\cap \Fix \mathscr S$. 
This yields 
$$
W^2_2(\mu,\nu)=W^2_2(\mathscr T(\mathscr S(\mu)),\mathscr T(\nu))\leq W^2_2(\mathscr S(\mu),\nu)
\leq W^2_2(\mu,\nu)-\frac{1-\alpha}{\alpha}W^2_2(\mu,\mathscr S(\mu)),
$$
which implies $\mathscr S(\mu) = \mu$, a contradiction.
\end{proof}

\begin{proposition}	\label{p:compositions}
	Let $\mathscr S:\mathcal P_2(\R^d)\to \mathcal P_2(\R^d)$ be quasi $\alpha$-firmly nonexpansive 
	and 
	let $\mathscr T:\mathcal P_2(\R^d)\to \mathcal P_2(\R^d)$ be quasi $\beta$-firmly nonexpansive.  
	Suppose that 
	$\Fix \mathscr T\cap \Fix \mathscr S\neq\emptyset$.   
	Then
	$\mathscr T \circ \mathscr S$ is quasi $\gamma$-firmly nonexpansive  with  
	\begin{equation}	\label{eq:composition}
	\gamma\coloneqq \frac{\alpha+\beta-2\alpha\beta}{1-\alpha\beta}.
	\end{equation}
\end{proposition}

\begin{proof} 
	By Lemma \ref{l:intersections}, assumption 
	$\Fix \mathscr T\cap\Fix \mathscr S\neq\emptyset$ implies that 
	$\Fix (\mathscr T \circ \mathscr S)=\Fix \mathscr T\cap\Fix \mathscr S$. 
	Let $\nu\in \Fix (\mathscr T \circ \mathscr S)$ and $\mu\in \mathcal P_2(\R^d)$. 
	Then an application of \eqref{eq:firmly} yields
	\begin{align*}
	W^2_2(\mathscr T(\mathscr S(\mu)),\nu)&\leq W_2^2(\mathscr S(\mu),\nu)-\frac{1-\beta}{\beta}W^2_2(\mathscr S(\mu),\mathscr T(\mathscr S(\mu)))\\&\leq W^2_2(\mu,\nu)-\frac{1-\alpha}{\alpha}W^2_2(\mu,\mathscr S(\mu))-\frac{1-\beta}{\beta}W^2_2(\mathscr S(\mu),\mathscr T(\mathscr S(\mu)).
	\end{align*}
	It suffices to show that 
	\begin{equation}
	\label{eq:condition}
	\frac{1-\gamma}{\gamma}W^2_2(\mu,\mathscr T(\mathscr S(\mu))\leq \frac{1-\alpha}{\alpha}W^2_2(\mu,\mathscr S(\mu))+\frac{1-\beta}{\beta}W^2_2(\mathscr S(\mu),\mathscr T(\mathscr S(\mu)).
	\end{equation}
With $\tau\coloneqq (1-\alpha)/\alpha+(1-\beta)/\beta$
inequality \eqref{eq:condition} is equivalent to
	\begin{align*}
	\Big(&\frac{1-\alpha}{\tau\alpha}\Big)^2W^2_2(\mu,\mathscr S(\mu))
	+\Big(\frac{1-\beta}{\tau\beta}\Big)^2W^2_2(\mathscr S(\mu),\mathscr T(\mathscr S(\mu))\\
	&+\Big(\frac{1-\alpha}{\tau\alpha}\Big)\Big(\frac{1-\beta}{\tau\beta}\Big)\Big(W^2_2(\mu,\mathscr S(\mu))
	+W^2_2(\mathscr S(\mu),\mathscr T(\mathscr S(\mu))-W^2_2(\mu,\mathscr T(\mathscr S(\mu))\Big)\geq 0.
	\end{align*}
Setting
	$\kappa\coloneqq \displaystyle\frac{1-\alpha}{\alpha}/\frac{1-\beta}{\beta}$,
this is equivalent to
	$$(\kappa+1)W^2_2(\mu,\mathscr S(\mu))+\frac{\kappa+1}{\kappa}W^2_2(\mathscr S(\mu),\mathscr T(\mathscr S(\mu))-W^2_2(\mu,\mathscr T(\mathscr S(\mu))\geq 0.$$
	The elementary inequality
	$$\kappa W^2_2(\mu,\mathscr S(\mu))+\frac{1}{\kappa}W^2_2(\mathscr S(\mu),\mathscr T(\mathscr S(\mu))\geq 2\,W_2(\mu,\mathscr S(\mu))W_2(\mathscr S(\mu),\mathscr T(\mathscr S(\mu))$$ for all $\kappa>0$
	together with the triangle inequality 
	$$W_2(\mu,\mathscr S(\mu))+W_2(\mathscr S(\mu),\mathscr T(\mathscr S(\mu))\geq W_2(\mu,\mathscr T(\mathscr S(\mu))$$
	implies
	\begin{align*}
	(\kappa+1)W^2_2(\mu,\mathscr S(\mu))&+\frac{\kappa+1}{\kappa}W^2_2(\mathscr S(\mu),\mathscr T(\mathscr S(\mu))\\
	&\geq (W_2(\mu,\mathscr S(\mu))+W_2(\mathscr S(\mu),\mathscr T(\mathscr S(\mu)))^2\geq W^2_2(\mu,\mathscr T(\mathscr S(\mu)). 
	\end{align*} 
\end{proof}

The following corollary is an immediate consequence of the above proposition.

\begin{corollary}
	Let $(\mathscr T_i)_{i=1}^N$ be a finite family 
	of quasi $\alpha_i$-firmly nonexpansive mappings. 
	Suppose that their fixed point sets have a nonempty intersection. 
	Then $\mathscr T=\mathscr T_{i_n} \circ \mathscr T_{i_{n-1}} \circ ... \circ \mathscr T_{i_1}$,
	where $i_j\in\{1,2,...,n\}$ are distinct, is also quasi $\alpha$-firmly nonexpansive for some $\alpha\in(0,1)$ dependent on $\alpha_i$.
\end{corollary}

\section{Fixed point theory of quasi $\alpha$-firmly nonexpansive mappings}\label{sec:fix}
\subsection{Opial's property and Fej\'er monotonicity}
A mapping $\mathscr T:\mathcal P_2(\R^d)\to\mathcal P_2(\R^d)$ is \emph{asymptotic regular at} $\mu\in\mathcal P_2(\R^d)$, 
if 
\begin{equation} \label{eq:asymptotic-regular}
\lim_{n\to+\infty} W_2(\mathscr T^{n+1}(\mu),\mathscr T^n(\mu))=0.
\end{equation}
Here $\mathscr T^n:=\underbrace{\mathscr T\circ\cdots\circ\mathscr T}_{n-\text{times}}$.  
If the limit in \eqref{eq:asymptotic-regular} holds for every $\mu\in\mathcal P_2(\R^d)$, 
then $\mathscr T$ is said to be \emph{asymptotic regular} on $\mathcal P_2(\R^d)$. 
An immediate consequence of quasi $\alpha$-firmly nonexpansiveness is the following lemma.

\begin{lemma}	\label{l:asympotic-reg}
	Let $\mathscr T:\mathcal P_2(\R^d)\to\mathcal P_2(\R^d)$ 
	be a quasi $\alpha$-firmly nonexpansive mapping. 
	Then $\mathscr T$ is asymptotic regular on $\mathcal P_2(\R^d)$.
\end{lemma}
\begin{proof}
	Let $\mathscr T:\mathcal P_2(\R^d)\to \mathcal P_2(\R^d)$ be a quasi $\alpha$-firmly nonexpansive mapping. 
	Then, by definition, $\Fix\mathscr T\neq\emptyset$ 
	and, for every $\nu\in\Fix\mathscr T$, we have
	$$W_2^2(\mathscr T^{n+1}(\mu),\nu)
	\leq 
	W^2_2(\mathscr T^n(\mu),\nu)-\frac{1-\alpha}{\alpha}\,W^2_2(\mathscr T^{n+1}(\mu),\mathscr T^n(\mu))
	\quad \text{for all } \mu\in\mathcal P_2(\R^d).
	$$ 
	In particular, $W_2(\mathscr T^{n+1}(\mu),\nu)\leq W_2(\mathscr T^n(\mu),\nu)$ implies 
	that $(W_2(\mathscr T^{n}(\mu),\nu)))_{n\in\mathbb N}$ is bounded and monotone non-increasing sequence in $\mathbb R$.
	Hence $\lim_{n\to+\infty}W_2(\mathscr T^{n}(\mu),\nu)=\ell(\nu)$ for a certain non-negative number $\ell(\nu)$. 
	Consequently, we get
	$$
	0\leq\lim_{n\to+\infty}W^2_2(\mathscr T^{n+1}(\mu),\mathscr T^n(\mu))
	\leq
	\frac{\alpha}{1-\alpha}\,\lim_{n\to+\infty}\Big( W^2_2(\mathscr T^n(\mu),\nu)- W^2_2(\mathscr T^{n+1}(\mu),\nu)\Big)=0.$$
\end{proof}

Let $(\mu_n)_{n\in\mathbb N}$ be a sequence in $\mathcal P_2(\R^d)$. 
An element $\mu\in\mathcal P_2(\R^d)$ is a \emph{narrow cluster point} of $(\mu_n)_{n\in\mathbb N}$ 
if and only if there exists a subsequence $(\mu_{n_k})_{k\in\mathbb N}$ such that $\mu_{n_k}\overset{\mathcal N}\to\mu$. 
Recently, it has been shown \cite[Theorem 5.1]{Naldi} that if $\mu_n\overset{\mathcal N}\to\mu$, 
then the following inequality holds true
\begin{equation} \label{eq:Opial}
\liminf_{n\to+\infty}W_2(\mu_n,\mu)<\liminf_{n\to+\infty}W_2(\mu_n,\nu),\quad \text{for all } \nu\in\mathcal P_2(\R^d)\setminus\{\mu\}.
\end{equation}
This is known as the \emph{Opial's property}. 
It implies, for all $\nu\in\mathcal P_2(\R^d)\setminus\{\mu\}$, that
\begin{equation}\label{eq:Opial-sup}
\limsup_{n\to+\infty}W_2(\mu_n,\mu)
=
\lim_{k\to+\infty}W_2(\mu_{n_k},\mu)<\liminf_{k\to+\infty}W_2(\mu_{n_k},\nu)\leq\limsup_{n\to+\infty}W_2(\mu_n,\nu).
\end{equation}

A sequence $(\mu_n)_{n\in\mathbb N}\subset\mathcal P_2(\R^d)$ is \emph{Fej\'er monotone with respect to a set} 
$S\subseteq\mathcal P_2(\R^d)$ if $W_2(\mu_{n+1},\nu) \leq W_2(\mu_n,\nu)$ for all $\nu\in S$ and for all $n\in\mathbb N$.

\begin{lemma}	\label{l:cluster-points}
	Let $(\mu_n)_{n\in\mathbb N}\subset\mathcal P_2(\R^d)$ be Fej\'er monotone with respect to a set $S\subseteq\mathcal P_2(\R^d)$. If all narrow cluster points of $(\mu_n)_{n\in\mathbb N}$ belong to $S$, 
	then $\mu_n\overset{\mathcal N}\to\mu$ for some $\mu\in\mathcal P_2(\R^d)$.
\end{lemma}

\begin{proof}
	Let $(\mu_n)_{n\in\mathbb N}\subset\mathcal P_2(\R^d)$ be Fej\'er monotone with respect to a set $S\subseteq \mathcal P_2(\R^d)$. 
	In particular, it follows that $(\mu_n)_{n\in\mathbb N}$ is bounded. 
	By Lemma \ref{l:fundamental} and Theorem \ref{th:Prokhorov}, 
	there exists a subsequence $(\mu_{n_k})_{k\in\mathbb N}$ of $(\mu_n)_{n\in\mathbb N}$ 
	narrowly converging to some element $\mu\in\mathcal P_2(\R^d)$. 
	Next, we prove that if $\nu\in S$ is another narrow cluster point of $(\mu_n)_{n\in\mathbb N}$, then $\mu=\nu$ 
	and the whole sequence $(\mu_n)_{n\in\mathbb N}$ narrowly converges to $\mu$. 
	Suppose on the contrary that $\mu\neq\nu$. 
	Let $\mu_{m_k}\overset{\mathcal N}\to\nu$. 
	Denote by $r_1=\limsup_{k\to+\infty}W_2(\mu_{n_k},\mu)$ 
	and $r_2=\limsup_{k\to+\infty}W_2(\mu_{m_k},\nu)$. 
	Assume w.l.o.g. that $r_1\leq r_2$. 
	By \eqref{eq:Opial-sup} 
	it follows that $r_2<\limsup_{k\to+\infty}W_2(\mu_{m_k},\mu)$. 
	For every $\varepsilon>0$, there is $k_0\in\mathbb N$ 
	such that $W_2(\mu_{n_k},\mu)<r_1+\varepsilon$ 
	whenever $k\geq k_0$. Moreover by Fej\'er monotonicity 
	$W_2(\mu_{m_k},\mu)<r_1+\varepsilon$ whenever $m_k\geq n_{k_0}$. 
	Consequently, there is $k_1\in\mathbb N$, 
	such that $W_2(\mu_{m_k},\mu)<r_2+\varepsilon$ whenever $k\geq k_1$. 
	Therefore $\limsup_{k\to+\infty}W_2(\mu_{m_k},\mu)\leq r_2$. 
	However, this contradicts Opial's property. 
	Hence the narrow cluster point is unique. 
	
	Now suppose that $\mu_n$ does not narrowly converge to $\mu\in S$. 
	Then there is a narrow open set $U$ containing $\mu$ such that $\mu_n\notin U$ 
	for infinitely many $n$. 
	Since $\{\mu_n\,:\,\mu_n\in \mathcal P(\R^d)\setminus U\}$ 
	is bounded, it possesses a narrowly convergent subsequence. 
	Let $\nu$ be the corresponding narrow cluster point. 
	Since $\mathcal P_2(\R^d)\setminus U$ is narrowly closed, 
	hence narrowly sequentially closed, we conclude that $\nu\in\mathcal P_2(\R^d)\setminus U$. 
	But by construction we have $\nu\neq\mu$, which 
	contradicts the uniqueness of the narrow cluster point.
\end{proof}

\subsection{Opial's Theorem}
A well-known result of Opial \cite{Opial} for uniformly convex Banach spaces $X$ satisfying Opial's property 
states that the iterations $x_{n+1}=Tx_n$ of a nonexpansive and asymptotic regular operator 
$T:X\to X$ with $\Fix T\neq\emptyset$ always converge weakly to an element in $\Fix T$. 
Recently, in \cite[Theorem 6.9]{Naldi}, it was shown that such a result holds true as well in the Wasserstein space 
$\mathcal P_2(\mathbb R^d)$ for mappings that are nonexpansive, asymptotic regular and have a nonempty fixed point set. 
This is because the space $\mathcal P_2(\mathbb R^d)$ satisfies Opial's property with respect to the narrow convergence. 


\begin{thm}\cite[Theorem 6.9]{Naldi} 	\label{th:Naldi}
	Let $\mathscr T:\mathcal  P_2(\mathbb R^d) \to\mathcal  P_2(\mathbb R^d)$ be a nonexpansive mapping that is asymptotic regular on 
	$\mathcal  P_2(\mathbb R^d)$. Then $\Fix \mathscr T\neq\emptyset$ if and only if for some  $\mu_0\in\mathcal  P_2(\mathbb R^d)$ (hence any $\mu_0\in\mathcal  P_2(\mathbb R^d)$) the iterates $\mu_{n+1}=\mathscr T(\mu_n)$ are bounded in $\mathcal  P_2(\mathbb R^d)$, in which case they narrowly converge to some $\mu\in\Fix\mathscr T$.
\end{thm}

As a consequence of this theorem we have the following corollary.

\begin{corollary}	\label{p:fixed-point-quasi}
	Let $\mathscr T:\mathcal  P_2(\mathbb R^d)\to\mathcal P_2(\mathbb R^d)$ be a nonexpansive,  quasi $\alpha$-firmly nonexpansive mapping. 
	Then, for any $\mu_0\in\mathcal  P_2(\mathbb R^d)$, 
	the iterates $\mu_{n+1}=\mathscr T(\mu_n)$ converge narrowly to some element $\mu\in\Fix\mathscr T$.
\end{corollary}

\begin{proof}
	By Lemma \ref{l:asympotic-reg} the operator $\mathscr T$ is asymptotic regular, 
	whenever it is quasi $\alpha$-firmly nonexpansive. 
	Moreover, by Definition \eqref{eq:firmly}, the fixed point set $\Fix \mathscr T$ is nonempty. 
	Hence the result follows directly from Theorem \ref{th:Naldi}. 
\end{proof}

A function $\phi:\mathcal P_2(\mathbb R^d)\to(-\infty,+\infty]$ is a \emph{characteristic function} 
of a mapping $\mathscr T:\mathcal P_2(\mathbb R^d)\to\mathcal P_2(\mathbb R^d)$, 
if 
$$\argmin\{\phi(\nu): \nu\in\mathcal P_2(\R^d)\}=\Fix \mathscr T$$ 
whenever the latter is nonempty. 
Let $\mathscr C_{\mathscr T}$ denote the set of all characteristic functions associated to the mapping $\mathscr T$. 
Note that $\mathscr C_{\mathscr T}\neq\emptyset$ since $\phi(\mu)=W_2(\mu,\mathscr T(\mu))$ 
is a characteristic function for any mapping $\mathscr T$ satisfying $\Fix\mathscr T\neq\emptyset$. 
A mapping $\mathscr T:\mathcal P_2(\R^d)\to\mathcal P_2(\mathbb R^d)$ is said to satisfy the \emph{quadratic growth condition}, 
if there exist a constant $C>0$ and a proper, narrow lower semicontinuous function $\phi\in\mathscr C_{\mathscr T}$ satisfying 
\begin{equation}\label{eq:quadratic}
W^2_2(\mathscr T(\mu),\nu)-W^2_2(\mu,\nu)\leq C\,(\phi(\nu)-\phi(\mathscr T(\mu))) \quad \text{for all } \mu,\nu\in\mathcal P_2(\mathbb R^d). 
\end{equation}

\begin{thm} \label{th:quadratic}
	Let $\mathscr T:\mathcal P_2(\mathbb R^d)\to\mathcal P_2(\mathbb R^d)$ be a quasi $\alpha$-firmly nonexpansive mapping 
	satisfying the quadratic growth condition \eqref{eq:quadratic}. Then, for any $\mu_0\in\mathcal P_2(\mathbb R^d)$, 
	the iterates $\mu_{n+1}=\mathscr T(\mu_n)$ converge narrowly to some element $\mu\in\Fix \mathscr T$.
\end{thm}

\begin{proof}
	Since $\mathscr T$ is quasi $\alpha$-firmly nonexpansive, then, by similar arguments as in Proposition \ref{p:fixed-point-quasi}, 
	it follows that, for any $\mu_0\in\mathcal P_2(\R^d)$, 
	the sequence $(\mu_n)_{n\in\mathbb N}$ defined by $\mu_{n+1}=\mathscr T(\mu_n)$ is bounded and therefore it contains a subsequence $(\mu_{n_k})_{k\in\mathbb N}$ narrowly converging to a certain element $\mu\in\mathcal P_2(\mathbb R^d)$. 
	The  assumption that $\mathscr T$ satisfies the quadratic growth condition implies 
	that there is a constant $C>0$ and a proper narrow lsc function $\phi\in\mathscr C_{\mathscr T}$ 
	such that inequality \eqref{eq:quadratic} is satisfied. 
	In particular, it follows that $\phi(\mathscr T(\mu))\leq\phi(\mu)$ for all 
	$\mu\in\mathcal P_2(\mathbb R^d)$. 
	Consequently, we obtain $\phi(\mu_{n+1})\leq\phi(\mu_n)$ for every $n\in\mathbb N$. 
	Again, from condition \eqref{eq:quadratic} for every $\nu\in\mathcal P_2(\mathbb R^d)$ it follows 
	\begin{align}
	\sum_{n=0}^N(W^2_2(\mathscr T(\mu_n),\nu)-W^2_2(\mu_n,\nu))
	&\leq C\,\sum_{n=0}^N(\phi(\nu)-\phi(\mathscr T(\mu_n))),\\
	W^2_2(\mathscr T(\mu_N),\nu)-W^2_2(\mu_0,\nu)
	&\leq C\,(N+1)\,\phi(\nu)-C\sum_{n=0}^N\phi(\mathscr T(\mu_n)).
	\end{align}
	Rearranging the terms and using the monotonicity of $(\phi(\mu_n))_{n\in\mathbb N}$ yields
	$$\phi(\mu_{N+1})+\frac{1}{C\,(N+1)}(W^2_2(\mu_{N+1},\nu)-W^2_2(\mu_0,\nu))\leq \phi(\nu) \quad \text{for all } \nu\in\mathcal P_2(\mathbb R^d).$$
	Consequently, by narrow lower semicontinuity of $\phi$, we get 
	$$\phi(\nu)\geq\limsup_{N\to+\infty}\phi(\mu_{N+1})\geq\liminf_{k\to+\infty}\phi(\mu_{n_k})\geq\phi(\mu)
	\quad \text{for all } \nu\in\mathcal P_2(\R^d).$$
	Therefore, $\mu\in\argmin\{\phi(\nu):\nu\in\mathcal P_2(\R^d)\}$. 
	
	Since $\phi$ is a characteristic function for the mapping $\mathscr T$ we have $\mu\in\Fix\mathscr T$. 
	If $\widetilde{\mu}\in\mathcal P_2(\R^d)$ is another narrow cluster point of $(\mu_n)_{n\in\mathbb N}$, 
	then by the same arguments we conclude that $\tilde{\mu}\in\Fix \mathscr T$. 
	By quasi $\alpha$-firmly nonexpansiveness of the mapping $\mathscr T$, 
	it holds that $(\mu_n)_{n\in\mathbb N}$ is Fej\'er monotone with respect to $\Fix\mathscr T$. 
	By Lemma \ref{l:cluster-points} it follows that the whole sequence $(\mu_n)_{n\in\mathbb N}$ narrowly converges to $\mu\in\Fix\mathscr T$.
\end{proof}

As a corollary, we obtain a result on the convergence of the so-called \emph{proximal point algorithm}
which is already known from the literature, see, e.g., \cite[Theorem 6.7]{Naldi} .

\begin{corollary}	\label{c:PPA}
	Let $\mathscr F:\mathcal P_2(\mathbb R^d)\to(-\infty+\infty]$ 
	be proper, lsc, coercive and convex along generalized geodesics. 
	For $\tau>0$, let $\mathscr J_{\tau}:\mathcal P_2(\R^d)\to\mathcal P_2(\mathbb R^d)$ 
	be the proximal mapping defined in \eqref{eq:proximal-mapping}. 
	Then, for any $\mu_0\in \overline{D(\mathscr F)}$, 
	the iterates $\mu_{n+1}=\mathscr J_{\tau}(\mu_n)$ converge narrowly to some 
	$\mu\in\argmin\{\mathscr F(\nu):\nu\in\mathcal P_2(\mathbb R^d)\}$. 
\end{corollary}

\begin{proof}
	By Proposition \ref{p:proximal-quasi}, the proximal mapping $\mathscr J_{\tau}$ is quasi $\alpha$-firmly nonexpansive with $\alpha=1/2$. 
	Moreover, from inequality \eqref{eq:1} we know that $\mathscr J_{\tau}$ satisfies the quadratic growth condition with $C=2\tau$ and $\phi=\mathscr F$. Then the result follows from Theorem \ref{th:quadratic}.
\end{proof}

\section{Cyclic proximal point algorithm} \label{sec:cppa}
Let $\mathscr F_i:\mathcal P_2(\R^d)\to(-\infty,+\infty]$ be proper, 
lsc, coercive and convex along generalized geodesics 
for $i=1,2,\cdots,N$. 
Consider the minimization problem
\begin{equation}\label{eq:minimization}
\inf_{\mu\in\mathcal P_2(\R^d)}\sum_{i=1}^N\mathscr F_i(\mu).
\end{equation}
The function $\mathscr F=\sum_{i=1}^N\mathscr F_i$ is itself proper, lsc, coercive and convex along generalized geodesics, 
since it is the sum of finitely many such functions. Suppose further that $D(\mathscr F)\subseteq D(\mathscr F_i)$.
A popular method to solve a problem of this kind is the proximal point algorithm from the previous section.
However, computing the proximal mapping $\mathscr J_{\tau}$ might be complicated because 
it could happen that the function $\mathscr F$ is difficult to handle, both theoretically and computationally. 
One way out consists in  considering the functions $\mathscr F_i$ separately, 
that is one computes the proximal mapping $\mathscr J_{\tau_i}$ for each function, 
where $\tau_i>0$ is the corresponding step size for $i=1,2,\cdots,N$. Then we consider the iterates
\begin{equation}
\label{eq:alg}
\mu_{n+1}=\mathscr J_{\tau_{[n]}}(\mu_n), \quad\text{where}\;[n]=n\,(\text{mod}\, N)+1\in\{1,2,\cdots, N\}.
\end{equation}
Such a method is known as the \emph{cyclic proximal point method}. 
For two operators, it is also called backward-backward splitting method.
Splitting methods in convex analysis date back to papers  of Lions, Mercier \cite[Lions--Mercier, 1979]{Lions} 
who studied splitting algorithms for stationary and evolution problems involving the sum of two monotone (multivalued) operators defined on a Hilbert space. 
In finite dimensional, linear spaces cyclic proximal point algorithms go back to \cite{ber2013}.
Since then splitting algorithms have been applied to more general problems in the setting of both linear and non-linear spaces. 
For example, in the context of complete $\CAT(0)$ spaces, this proximal point algorithms were studied in \cite{Bac13}, see also \cite{FO02}
and their cyclic version in \cite{Bac14}. 
In this paper, we introduce the cyclic proximal point algorithm in $\mathcal P_2(\R^d)$. For recent papers on related algorithms, see e.g. \cite{FTC2021,SKL2021}.

\begin{thm}	\label{th:SPPA}
	Let $\mathscr F_i:\mathcal P_2(\R^d)\to(-\infty,+\infty]$ be proper, lsc, coercive and convex along generalized geodesics. Denote by $\mathscr F=\sum_{i=1}^N\mathscr F_i$ and suppose that $\emptyset\neq D(\mathscr F)\subseteq D(\mathscr F_i)$ for $i=1,2,\cdots,N$.
	Let $\mathscr J_{\tau_i}:\mathcal P_2(\R^d)\to\mathcal P_2(\R^d)$ 
	be the proximal mapping of $\mathscr F_i$ for $i=1,2,\cdots,N$. 
	Assume that $\bigcap_{i=1}^N\Fix\mathscr J_{\tau_i}\neq\emptyset$. 
	Then, for any $\mu_0\in\overline{D(\mathscr F)}$, the iterates 
	$\mu_{n+1}=\mathscr J_{[n]}(\mu_n)$ 
	converge narrowly to a solution of \eqref{eq:minimization}. 
\end{thm}

\begin{proof}
Let $\mathscr J\coloneqq \mathscr J_{\tau_N} \circ \mathscr J_{\tau_{N-1}} \circ \cdots \circ \mathscr J_{\tau_1}$. 
Assumption  $\bigcap_{i=1}^N\Fix\mathscr J_{\tau_i}\neq\emptyset$ implies by Lemma \ref{l:intersections} 
that $\Fix\mathscr J =\bigcap_{i=1}^N\Fix\mathscr J_{\tau_i}$. 
Given $\mu_0\in\overline{D(\mathscr F)}$, we define $\mu_{n+1}=\mathscr J_{\tau_{[n]}}(\mu_n)$. 
By Proposition \ref{p:proximal-quasi}, 
the mapping $\mathscr J_{\tau_i}$ is quasi $1/2$-firmly nonexpansive for every $i=1,2,\cdots,N$ 
and in particular $\mathscr J_{\tau_i}$ is quasi nonexpansive for every $i=1,2,\cdots,N$. 
Therefore, we get for every $\nu\in\Fix\mathscr J$ that
$$W_2(\mu_{n+1},\nu)=W_2(\mathscr J_{\tau_{[n]}}(\mu_n),\nu)\leq W_2(\mu_n,\nu).$$
Consequently, the sequence $(\mu_{n})_{n\in\mathbb N}$ is bounded. 
Hence,  by Lemma \ref{l:fundamental} and Prokhorov's Theorem \ref{th:Prokhorov} 
it has a subsequence $(\mu_{n_k})_{k\in\mathbb N}$ narrowly converging to some measure $\mu\in\mathcal P_2(\R^d)$. 
Since there are finitely many indices $i\in\{1,2,\cdots,N\}$, 
we get by the pigeonhole principle that
$\mu_{n_k}=\mathscr J_{\tau_j}(\mu_{n_k-1})$ for infinitely many $k\in\mathbb N$, for some $j\in\{1,2,\cdots,N\}$.
Moreover, by inequality \eqref{eq:1}, we have for all $l\in\mathbb N$ and all $\nu\in D(\mathscr F)$ that
$$\frac{1}{2\tau_j}W^2_2(\mu_{n_{k(l)}},\nu)-\frac{1}{2\tau_j}W^2_2(\mu_{n_{k(l)}-1},\nu)
\leq \mathscr F_j(\nu)-\mathscr F_j(\mu_{n_{k(l)}}).$$
From Fej\'er monotonicity, we get $W_2(\mu_{n_{k(l)}},\nu)\geq W_2(\mu_{n_{{k(l+1)}}-1},\nu)$ 
for any $\nu\in\Fix\mathscr J$ and every $l\in\mathbb N$. 
Then, rearranging terms in the last inequality, yields for all $l\in\mathbb N$ and all $\nu\in \Fix\mathscr J$ that
\begin{equation*}
\label{eq:3}\mathscr F_j(\mu_{n_{k(l)}})
\leq 
\mathscr F_j(\nu)+\frac{1}{2\tau_j}W^2_2(\mu_{n_{k(l-1)}},\nu)-\frac{1}{2\tau_j}W^2_2(\mu_{n_{k(l)}},\nu).
\end{equation*}
Passing to the limit as $l\to+\infty$ and from Fej\'er monotonicity of $(\mu_n)_{n\in\mathbb N}$ with respect to $\Fix\mathscr J$, 
we obtain for all $\nu\in\Fix\mathscr J$ that
$$
\liminf_{l\to+\infty}\mathscr F_j(\mu_{n_{k(l)}})
\leq 
\mathscr F_j(\nu)+\frac{1}{2\tau_j}\lim_{l\to+\infty}\Big(W^2_2(\mu_{n_{k(l-1)}},\nu)-W^2_2(\mu_{n_{k(l)}},\nu)\Big)=\mathscr F_j(\nu).
$$
On the other hand, 
$\Fix\mathscr J\subseteq\Fix\mathscr J_{\tau_j}=\argmin\{\mathscr F_j(\nu):\nu\in\mathcal P_2(\R^d)\}$ 
implies that the last inequality holds for all 
$\nu\in\argmin\{\mathscr F_j(\sigma):\sigma\in\mathcal P_2(\R^d)\}$ 
and so for all $\nu\in\mathcal P_2(\R^d)$.
Narrow lsc of $\mathscr F_j$ gives
$$\mathscr F_j(\mu)\leq \liminf_{l\to+\infty}\mathscr F_j(\mu_{n_{k(l)}})\leq\mathscr F_j(\nu) 
\quad  \text{for all }  \nu\in\mathcal P_2(\R^d),
$$ 
and therefore $\mu\in\argmin\{\mathscr F_j(\nu):\nu\in\mathcal P_2(\R^d)\}$. 
Now consider the sequence $(\mu_{n_{k(l)-1}})_{l\in\mathbb N}$ that 
by construction satisfies $\mu_{n_{k(l)-1}}=\mathscr J_{\tau_{j-1}}(\mu_{n_{k(l)-2}})$. 
Since $(\mu_{n_{k(l)-1}})_{l\in\mathbb N}$ is bounded,
let $\mu_{n_{k(l)-1}}\overset{\mathcal N}\to \mu'\in\mathcal P_2(\R^d)$, 
else by Lemma \ref{l:fundamental} and Prokhorov's Theorem \ref{th:Prokhorov}, 
we can always pass to a subsequence of $(\mu_{n_{k(l)-1}})_{l\in\mathbb N}$ 
with this property. 
By similar arguments as above, we find that the limit 
$\mu'\in\argmin\{\mathscr F_{j-1}(\nu):\nu\in\mathcal P_2(ßR^d)\}$. 
By inequality \eqref{eq:1} we have 
$$\frac{1}{2\tau_j}W^2_2(\mu_{n_{k(l)}},\mu_{n_{k(l)-1}})+\frac{1}{2\tau_j}W^2_2(\mu_{n_{k(l)}},\nu)-\frac{1}{2\tau_j}W^2_2(\mu_{n_{k(l)-1}},\nu)\leq\mathscr F_j(\nu)-\mathscr F_j(\mu_{n_{k(l)}})$$
for all $\nu\in D(\mathscr F)$ and in particular for all $\nu\in\Fix\mathscr J$. 
From narrow lsc of $W(\cdot,\cdot)$ and Fej\'er monotonicity of $(\mu_n)_{n\in\mathbb N}$ with respect to $\Fix\mathscr J$, 
passing to the limit as $l\to+\infty$, 
we obtain
$$\frac{1}{2\tau_j}W^2_2(\mu,\mu')\leq \mathscr F_j(\nu)-\mathscr F_j(\mu) \quad \text{for all } \nu\in\Fix\mathscr J.$$
Since $\Fix\mathscr J\subseteq \argmin\{\mathscr F_j(\nu):\nu\in\mathcal P_2(\R^d)\}$,
the last inequality holds in particular for all 
$\nu\in\argmin\{\mathscr F_j(\nu):\nu\in\mathcal P_2(ßR^d)\}$. 
Therefore $W_2(\mu,\mu')\leq 0$ implies that $\mu=\mu'$. 
This means that $\mu\in\argmin\{\mathscr F_{j-1}(\nu):\nu\in\mathcal P_2(\R^d)\}$. 
Repeating the same argument for every index $i\in\{1,2,\cdots,N\}$ 
yields that $\mu\in\argmin\{\mathscr F_i(\nu):\nu\in\mathcal P_2(\R^d)\}$ for all $i=1,2,\cdots,N$, 
so that
$$
\mu\in\bigcap_{i=1}^N\argmin\{\mathscr F_i(\nu)\,:\,\nu\in\mathcal P_2(\R^d)\}
\subseteq \argmin\{\mathscr F(\nu)\,:\,\nu\in\mathcal P_2(\R^d)\}.
$$ 
Hence we obtain for the original subsequence 
$\mu_{n_k}\overset{\mathcal N}\to\mu\in \bigcap_{i=1}^N\argmin\{\mathscr F_i(\nu):\nu\in\mathcal P_2(\R^d)\}$. 
If $\mu'$ is another narrow cluster point of $(\mu_n)_{n\in\mathbb N}$, 
then, by same arguments, we obtain  
$$\mu'\in \bigcap_{i=1}^N\argmin\{\mathscr F_i(\nu)\,:\,\nu\in\mathcal P_2(\R^d)\}.$$ 
We have that $\bigcap_{i=1}^N\argmin\{\mathscr F_i(\nu)\,:\,\nu\in\mathcal P_2(\R^d)\}=\Fix\mathscr J$ 
and that the sequence $(\mu_n)_{n\in\mathbb N}$ is Fej\'er monotone with respect to $\Fix\mathscr J$. 
Consequently, by Lemma \ref{l:cluster-points}, the whole sequence converges $\mu_n\overset{\mathcal N}\to\mu\in \Fix\mathscr J$. 
This completes the proof.
\end{proof}

Since the our theory  relies on the assumption that the intersection of the fixed points sets of a finite collection of quasi $\alpha$-firmly nonexpansive operators is nonempty, the last results cannot be applied to a situation when this intersection is empty. 
However, inspired by a result of Ba\v cak \cite[Theorem 6.3.7]{Bacak}, 
we can provide a convergence theorem, when the functions $\mathscr F_i$ do not have a common minimizer, 
which essentially is the case when the corresponding proximal mappings $\mathscr J_{\tau_i}$ have no common fixed point. 
However, we need to add two conditions. 
First, each function $\mathscr F_i$ is Lipschitz continuous on $D(\mathscr F_i)$. 
This means that there exists $L_i>0$ such that 
$|\mathscr F_i(\mu)-\mathscr F_i(\nu)|\leq L_i\, W_2(\mu,\nu)$ for all $\mu,\nu\in D(\mathscr F_i)$. 
Second, if $\mathscr J_{i,\tau_{k}}$ is the proximal mapping of $\mathscr F_i$ with step size $\tau_{k}$, 
we require that $(\tau_{k})_{k\in\mathbb N_0}$ satisfies $\sum_{k\in\mathbb N}\tau_{k}=+\infty$ 
and $\sum_{k\in\mathbb N_0}\tau^2_{k}<+\infty$. 
Then we consider the iterations
\begin{equation}
\label{eq:SPPA-2}
\mu_{kN+n+1}=\mathscr J_{[n],\tau_k}(\mu_{kN+n}),\quad  [n]=n\,(\text{mod}\, N)+1\in\{1,2,\cdots, N\},\;k=0,1,2,\cdots.
\end{equation}

\begin{thm}
	\label{th:SPPA-2}
	Let $\mathscr F_i:\mathcal P_2(\R^d)\to(-\infty,+\infty]$ be proper, lsc, coercive and convex along generalized geodesics. Denote by $\mathscr F=\sum_{i=1}^N\mathscr F_i$ and suppose that $\emptyset\neq D(\mathscr F)\subseteq D(\mathscr F_i)$ for $i=1,2,\cdots,N$.
	Assume that $\mathscr F_i$ is Lipschitz continuous on $D(\mathscr F_i)$ for every $i\in\{1,2,\cdots,N\}$. 
	Denote by $\mathscr J_{i,\tau_k}$  the proximal mapping of $\mathscr F_i$ with step size $\tau_{k}>0$ 
	satisfying $\sum_{k\in\mathbb N_0}\tau_{k}=+\infty$ and $\sum_{k\in\mathbb N_0}\tau^2_{k}<+\infty$. 
	Then, for any initial measure $\mu_0\in \overline{D(\mathscr F)}$, 
	the iterates $\mu_{kN+n+1}=\mathscr J_{[n],\tau_k}(\mu_{kN+n})$ 
	converge narrowly to a solution of problem \eqref{eq:minimization}.
\end{thm}

The proof follows similar steps as in \cite[Theorem 6.3.7]{Bacak}.

\begin{proof}
	First, we get from inequality \eqref{eq:1} for each $i\in\{1,2,\cdots,N\}$ that 
	$$W^2_2(\mu_{kN+i},\nu)\leq W^2_2(\mu_{kN+i-1},\nu)-2\tau_{k}(\mathscr F_i(\mu_{kN+i})-\mathscr F_i(\nu))
	\quad \text{for all } \nu\in D(\mathscr F).$$
	Summing on the both sides of this inequality yields
	\begin{align}
	W^2_2(\mu_{kN+N},\nu)&\leq W^2_2(\mu_{kN},\nu)-2\tau_{k}\sum_{i=1}^N\mathscr{F}_i(\mu_{kN+i})+2\tau_k\mathscr{F}(\nu)\\
	&=W^2_2(\mu_{kN},\nu)-2\tau_k(\mathscr F(\mu_{kN})-\mathscr F(\nu))
	+2\tau_k\mathscr F(\mu_{kN})-2\tau_{k}\sum_{i=1}^N\mathscr{F}_i(\mu_{kN+i}). \label{xyz}
	\end{align}
The assumption that $\mathscr F_i$ is Lipschitz continuous on $D(\mathscr F_i)$ and hence on $D(\mathscr F)$ for every $i\in\{1,2,\cdots,N\}$ 
implies for all $k\in\mathbb N_0$ that
$$
\mathscr F(\mu_{kN})-\sum_{i=1}^N\mathscr{F}_i(\mu_{kN+i})
=
\sum_{i=1}^N(\mathscr{F}_i(\mu_{kN})-\mathscr{F}_i(\mu_{kN+i}))\leq \sum_{i=1}^NL_i\,W_2(\mu_{kN},\mu_{kN+i}).
$$
By definition of the proximum we have for all $i\in\{1,2,\cdots,N\}$ and all $k\in\mathbb N_0$ that
$$
\mathscr F_i(\mu_{kN+i})+\frac{1}{2\tau_k}\,W^2_2(\mu_{kN+i-1},\mu_{kN+i})\leq \mathscr F_i(\mu_{kN+i-1}).
$$
This implies 
$$W_2(\mu_{kN+i-1},\mu_{kN+i})\leq2\tau_k\frac{\mathscr F_i(\mu_{kN+i-1})-\mathscr F_i(\mu_{kN+i})}{W_2(\mu_{kN+i-1},\mu_{kN+i})}
\leq 2\tau_k\,L_i.$$
This upper estimate together with an iterative application of the triangle inequality to the expression 
$W_2(\mu_{kN},\mu_{kN+i})\leq W_2(\mu_{kN},\mu_{kN+1})+\cdots+W_2(\mu_{kN+i-1},\mu_{kN+i})$ yields
$$
\sum_{i=1}^N\mathscr{F}_i(\mu_{kN})-\mathscr{F}_i(\mu_{kN+i})
\leq 
2\tau_k  \sum_{i=1}^NL_i\sum_{j=1}^iL_j\leq \tau_k L^2_{\max} N(N+1),
$$
where $L_{\max}\coloneqq\max\{L_i:i=1,2,\cdots,N\}$. 
Therefore, we obtain for all $\nu\in D(\mathscr F)$ the inequality 
\begin{equation} \label{eq:almsot-Fejer}
W^2_2(\mu_{kN+N},\nu)\leq W^2_2(\mu_{kN},\nu)-2\tau_k(\mathscr F(\mu_{kN})-\mathscr F(\nu))+2\tau^2_k\,L^2_{\max}
N(N+1).
\end{equation}
In particular, \eqref{eq:almsot-Fejer} holds if $\nu\in\argmin\{\mathscr F(\sigma):\sigma\in\mathcal P_2(\R^d)\}$. 
Applying \cite[Exercise 6.5]{Bacak} with 
$a_k\coloneqq W^2_2(\mu_{kN},\nu)$, $b_k\coloneqq \mathscr F(\mu_{kN})-\mathscr F(\nu)$ 
and $c_k\coloneqq2\tau^2_k\,L^2_{\max}N(N+1)$ 
yields that the sequence $(W_2(\mu_{kN},\nu))_{k\in\mathbb N_0}$ 
converges to a certain number $d(\nu)\geq 0$. 
In particular, the sequence $(\mu_{kN})_{k\in\mathbb N_0}$ is bounded.
By Lemma \ref{l:fundamental} and Prokhorov's Theorem \ref{th:Prokhorov}, 
there is a subsequence $\mu_{k_jN}\overset{\mathcal N}\to\mu$ for some $\mu\in\mathcal P_2(\R^d)$. 
On the other hand, again by \cite[Exercise 6.5]{Bacak}, 
it holds that 
$$\sum_{k\in\mathbb N_0}\tau_k(\mathscr F(\mu_{kN})-\mathscr F(\nu))<+\infty\quad \text{for all } 
\nu\in\argmin\{\mathscr F(\sigma):\sigma\in\mathcal P_2(\R^d)\}.
$$ 
Therefore $\lim_{k\to+\infty}\tau_k(\mathscr F(\mu_{kN})-\mathscr F(\nu))=0$ 
implies that $\lim_{k\to+\infty}\mathscr F(\mu_{kN})=\mathscr F(\nu)$, 
else we can always pass to a subsequence of $(\mu_{kN})_{k\in\mathbb N_0}$ having this property.
By narrow lsc of $\mathscr F$ we get that 
$$\mathscr F(\mu)\leq\liminf_{j\to+\infty}\mathscr F(\mu_{k_jN})\leq\limsup_{k\to+\infty}\mathscr F(\mu_{kN})=\mathscr F(\nu) 
$$
for all 
$\nu\in\argmin\{\mathscr F(\sigma):\sigma\in\mathcal P_2(\R^d)\}$.

Thus, 
$\mu \in \argmin \{ \mathscr F(\sigma): \sigma \in \mathcal P_2(\R^d) \}$. 
Let $(\mu_{k_mN})_{m\in\mathbb N}$ be another narrowly convergent subsequence of $(\mu_{kN})_{k\in\mathbb N_0}$. 
Let $\mu_{k_mN}\overset{\mathcal N}\to\mu'\in\mathcal P_2(\R^d)$. 
Note that \eqref{eq:almsot-Fejer} acts as a substitute in the argument of Lemma \ref{l:cluster-points} 
for Fej\'er monotonicity of $(\mu_{kN})_{k\in\mathbb N_0}$ with respect to 
$\argmin\{\mathscr F(\sigma):\sigma\in\mathcal P_2(\R^d) \}$. 
Indeed, let $r_1=\limsup_{j\to+\infty}W_2(\mu_{k_jN},\mu)$ 
and 
$r_2=\limsup_{m\to+\infty}W_2(\mu_{k_mN},\mu')$. 
Suppose w.l.o.g. that $r_1\leq r_2$. 
From Opial's property \eqref{eq:Opial-sup} 
it follows that $r_2<\limsup_{m\to+\infty}W_2(\mu_{k_mN},\mu)$. 
For every $\varepsilon>0$, there is $j_0\in\mathbb N$ such that $W_2(\mu_{k_jN},\mu)<r_1+\varepsilon$, 
whenever $j\geq j_0$. 
In \eqref{eq:almsot-Fejer}, let $\varepsilon^2_k\coloneqq 2\tau_k^2 L^2_{\max}N(N+1)$. 
Then we have 
$W^2_2(\mu_{k_mN},\mu)<(r_1+\varepsilon)^2+\sum_{l=k_{j_0}}^{k_m}\varepsilon^2_l$
whenever $k_m\geq k_{j_0}$.  
For a fixed difference $\Delta(m,j_0)=k_m-k_{j_0}$, we let $j_0\to+\infty$, i.e., also $m\to+\infty$. 
Since $\varepsilon_l\to 0$ and the sum $\sum_{l=k_{j_0}}^{k_m}\varepsilon^2_l$ 
is finite, we get  for sufficiently large $j_0$ and sufficiently large $m$ 
that $W_2(\mu_{k_mN},\mu)<r_1+2\varepsilon$. 
Therefore there exists $m_1\in\mathbb N$ such that $W_2(\mu_{k_mN},\mu)<r_2+2\varepsilon$, 
whenever $m\geq m_1$, implying $\limsup_{m\to+\infty}W_2(\mu_{k_mN},\mu)\leq r_2$. 
This would raise a contradiction. 
Therefore the sequence $(\mu_{kN})_{k\in\mathbb N_0}$ must have a unique narrow cluster point. 
Following exactly the same arguments as in Lemma \ref{l:cluster-points}, 
we get that the whole sequence converges $\mu_{kN}\overset{\mathcal N}\to\mu$. 
Now consider $(\mu_{kN+i})_{k\in\mathbb N_0}$ for $i=1,2,\cdots, N-1$. 
Repeating the same reasoning as for $(\mu_{kN})_{k\in\mathbb N_0}$, 
we conclude that 
$\mu_{kN+i}\overset{\mathcal N}\to \mu_i\in\argmin\{\mathscr F(\sigma):\sigma\in\mathcal P_2(\R^d) \}$ 
for each $i=1,2,\cdots,N-1$. 
From the estimate $W_2(\mu_{kN},\mu_{kN+i})\leq 2\tau_k\sum_{j=1}^iL_j$ 
and narrow lsc of $W_2(\cdot,\cdot)$, 
we obtain that 
$$0\leq W_2(\mu,\mu_i)\leq\liminf_{k\to+\infty}W_2(\mu_{kN},\mu_{kN+i})\leq\lim_{k\to+\infty}(2\tau_k\sum_{j=1}^iL_j)=0.$$ 
Hence $\mu=\mu_i$ for every $i=1,2,\cdots, N-1$. 
This means that the whole sequence of iterates $\mu_{kN+n+1}=\mathscr J_{[n],\tau_k}(\mu_{kN+n})$ 
converges narrowly to $\mu\in\argmin\{\mathscr F(\sigma)\,:\,\sigma\in\mathcal P_2(\R^d)\}$.
\end{proof}

\paragraph{Acknowledgment:}
This research was supported by DFG under Germany's Excellence Strategy – The Berlin Mathematics Research Center MATH+ (EXC-2046/1,  Projektnummer:  390685689).

\bibliographystyle{abbrv}
\bibliography{references}

\begin{thebibliography}{10}

\bibitem{alex}
A.~D. Aleksandrov.
\newblock A theorem on triangles in a metric space and some of its
  applications.
\newblock {\em Trudy Mat. Inst. Steklov.}, 38:5--23, 1951.

\bibitem{ABS21}
L.~Ambrosio, E.~Bru\'{e}, and D.~Semola.
\newblock {\em Lectures on Optimal Transport}.
\newblock Number 130 in Unitext. Springer, 2021.

\bibitem{Ambrosio}
L.~Ambrosio, N.~Gigli, and G.~Savar\'e.
\newblock {\em Gradient {F}lows in {M}etric {S}paces and in the {S}pace of
  {P}robability {M}easures}.
\newblock Lectures in Mathematics ETH Z\"urich. Birkh\"auser, 2005.

\bibitem{Bac13}
M.~Ba{\v c}{\'a}k.
\newblock {The proximal point algorithm in metric spaces}.
\newblock {\em Isr. J. Math.}, 194(2):689--701, 2013.

\bibitem{Bac14}
M.~Ba{\v{c}}{\'a}k.
\newblock Computing medians and means in {H}adamard spaces.
\newblock {\em SIAM J. Optim.}, 24(3):1542--1566, 2014.

\bibitem{Bacak}
M.~Ba{\v c}{\'a}k.
\newblock {\em Convex {A}nalysis and {O}ptimization in {H}adamard {S}paces},
  volume 22 of De Gruyter Series in Nonlinear Analysis and Applications.
\newblock De Gruyter, Berlin, 2014.

\bibitem{BacBerSteWei16}
M.~Ba{\v c}{\'a}k, R.~Bergmann, G.~Steidl, and A.~Weinmann.
\newblock A second order non-smooth variational model for restoring
  manifold-valued images.
\newblock {\em SIAM J. Sci. Comput.}, 38(1):567--597, 2016.

\bibitem{BC11}
H.~H. Bauschke and P.~L. Combettes.
\newblock {\em Convex Analysis and Monotone Operator Theory in Hilbert Spaces}.
\newblock Springer, New York, 2011.

\bibitem{Berdellima}
A.~B\"erd\"ellima.
\newblock On a notion of averaged mappings in {CAT}(0) spaces.
\newblock {\em Funct. Anal. its Appl.}, 56(1):37--50, 2022.

\bibitem{Luke}
A.~B\"erd\"ellima, F.~Lauster, and D.~Luke.
\newblock $\alpha$-{F}irmly nonexpansive operators on metric spaces.
\newblock {\em J. Fixed Point Theory Appl.}, 24(1), 2022.

\bibitem{Berd-Steidl}
A.~B\"erd\"ellima and G.~Steidl.
\newblock On $\alpha$-firmly nonexpansive operators in $r$-uniformly convex
  spaces.
\newblock {\em Results Math.}, 76, 2021.

\bibitem{BerPerSte16}
R.~Bergmann, J.~Persch, and G.~Steidl.
\newblock A parallel {{Douglas}}-{{Rachford}} algorithm for minimizing
  {{ROF}}-like functionals on images with values in symmetric {{Hadamard}}
  manifolds.
\newblock {\em SIAM J. Imag. Sci.}, 9(3):901--937, 2016.

\bibitem{ber2013}
D.~P. Bertsekas.
\newblock Incremental proximal methods for large scale convex optimization.
\newblock {\em Math. Program., Ser. B}, 129(2):163--195, 2011.

\bibitem{BPCPE11}
S.~Boyd, N.~Parikh, E.~Chu, B.~Peleato, and J.~Eckstein.
\newblock Distributed optimization and statistical learning via the alternating
  direction method of multipliers.
\newblock {\em Found. Trends Mach. Learn.}, 3(1):101--122, 2011.

\bibitem{Browder}
F.~E. Browder.
\newblock Convergence theorems for sequences of nonlinear operators in {B}anach
  spaces.
\newblock {\em Math. Zeitschr.}, 100:201--225, 1967.

\bibitem{Bruck}
R.~E. Bruck~Jr.
\newblock Nonexpansive projections on subsets of {B}anach spaces.
\newblock {\em Pacific J. Math.}, 47:341--355, 1973.

\bibitem{CP2019}
M.~Cuturi and G.~Peyr\'e.
\newblock Computational optimal transport.
\newblock {\em Found. Trends Mach. Learn.}, 11(5-6):355--607, 2019.

\bibitem{FTC2021}
J.~Fan, A.~Taghvaei, and Y.~Chen.
\newblock Variational {W}asserstein gradient flow.
\newblock {\em arXiv: 2112.02424}, 2021.

\bibitem{FO02}
O.~P. Ferreira and P.~R. Oliveira.
\newblock {Proximal point algorithm on {R}iemannian manifolds}.
\newblock {\em Optim.}, 51(2):257--270, 2002.

\bibitem{JKO1998}
R.~Jordan, D.~Kinderlehrer, and F.~Otto.
\newblock The variational formulation of the fokker-planck equation.
\newblock {\em SIAM J. Math. Anal.}, 29:1--17, 1998.

\bibitem{Krasnoselski}
M.~A. Krasnoselskij.
\newblock Two remarks on the method of successive approximations.
\newblock {\em Uspehi. Mat. Nauk (N.S.)}, 10:123--127, 1955.

\bibitem{Lions}
P.~L. Lions and B.~Mercier.
\newblock Splitting algorithms for the sum of two nonlinear operators.
\newblock {\em SIAM J. Numer. Anal.}, 16(6):964--979, 1979.

\bibitem{LukTamTha18}
D.~R. Luke, N.~H. Thao, and M.~K. Tam.
\newblock Quantitative convergence analysis of iterated expansive, set-valued
  mappings.
\newblock {\em Math. Oper. Res.}, 43(4):1143--1176, 2018.

\bibitem{Mann}
W.~R. Mann.
\newblock Mean value methods in iteration.
\newblock {\em Proc. Amer. Math. Soc.}, 4:506--510, 1953.

\bibitem{Naldi}
E.~Naldi and G.~Savar\'e.
\newblock Weak topology and opial property in {W}asserstein spaces, with
  applications to gradient flows and proximal point algorithms of geodesically
  convex functionals.
\newblock {\em Atti Accad. Naz. Lincei Cl. Sci. Fis. Mat. Natur.},
  32(4):725--750, 2022.

\bibitem{Opial}
Z.~Opial.
\newblock Weak convergence of the sequence of successive approximations for
  nonexpansive mappings.
\newblock {\em Bull. Amer. Math. Soc.}, 73:591--597, 1967.

\bibitem{Prokhorov}
Y.~V. Prokhorov.
\newblock Convergence of random processes and limit theorems in probability
  theory.
\newblock {\em Theory Probab. Its Appl.}, 1(2):157--214 (Russian), 1956.

\bibitem{SKL2021}
A.~Salim, A.~Korba, and G.~Luise.
\newblock The {W}asserstein proximal gradient algorithm.
\newblock {\em arXiv: 200203035}, 2021.

\bibitem{Santambrosio2015}
F.~Santambrogio.
\newblock {\em Optimal Transport for Applied Mathematicians}.
\newblock Springer, 2015.

\bibitem{S2017}
F.~Santambrogio.
\newblock { Euclidean, Metric, and Wasserstein } gradient flows: an overview.
\newblock {\em Bull. Math. Sci.}, 7:87--154, 2017.

\bibitem{Villani}
C.~Villani.
\newblock {\em Optimal {T}ransport: {O}ld and {N}ew}.
\newblock Grundlehren der mathematischen Wissenschaften 338. Springer-Verlag
  Berlin Heidelberg, 2009.

\bibitem{Yue}
M.~Yue, D.~Kuhn, and W.~Wiesemann.
\newblock On linear optimization over {W}asserstein balls.
\newblock {\em Math. Program.}, 2021.

\end{thebibliography}

\end{document}